\documentclass[]{article}

\usepackage{geometry}
\usepackage{booktabs}
\usepackage{graphicx}
\usepackage{placeins}

\usepackage{mathtools} 
\usepackage{mathrsfs}
\usepackage{amsmath}
\usepackage{amsfonts}
\usepackage{amssymb}
\usepackage{xcolor}

\usepackage{amsthm}
\theoremstyle{definition}

\newtheorem{thm}{Theorem}[section]
\newtheorem{prop}[thm]{Proposition}
\newtheorem{exmp}[thm]{Example}
\newtheorem{cor}[thm]{Corollary}
\newtheorem{lemma}[thm]{Lemma}

\newtheorem{defn}[thm]{Definition}
\newtheorem{rmk}[thm]{Remark}

\usepackage{tikz}
\usetikzlibrary{arrows.meta}
\tikzset{
	observed/.style = {circle, draw, fill = black, minimum size = 6pt, inner sep = 0pt},
	forced/.style = {circle, draw, fill = white, minimum size = 8pt, inner sep = 0pt, font = \scriptsize}
}
\usepackage{float}

\usepackage{csquotes}
\usepackage[
backend=biber,
style=numeric,
sorting=none,
doi=true,
url=false,
isbn=false
]{biblatex}
\usepackage{hyperref}
\hypersetup{hidelinks}

\newcommand{\diag}{\text{diag}}

\newcommand{\Disc}{\text{Disc}}
\newcommand{\rank}{\text{rank}}
\newcommand{\linspan}{\text{span}}

\newcommand{\norm}[1]{\left\lVert#1\right\rVert}

\title{Simultaneous Recovery of Wave Speeds and Initial Sources on Graphs via Redundant Frontier Propagation}
\author{Amir Moradifam\footnote{Department of Mathematics, University of California, Riverside, California, USA. E-mail: amirm@ucr.edu} \and Oliver P. Waite\footnote{Department of Mathematics, University of California, Riverside, California, USA. E-mail: oliver.waite@email.ucr.edu}}
\date{}

\begin{document}
	
	\maketitle
	
	\begin{abstract}
		We study the simultaneous recovery of an unknown initial signal and unknown vertex wave speeds for the wave equation on a finite graph from time-dependent measurements on a prescribed set of vertices. With the graph geometry known, we first show that a single experiment determines both quantities under a new two-witness zero forcing condition and generic nondegeneracy assumptions. We then develop a more general frontier-propagation framework in which the redundancy required for simultaneous recovery is distributed across successive graph contours rather than localized at individual forcing vertices. We derive Lipschitz stability estimates for the frontier propagation problem. We also show the recovery and stability results extend to positively weighted graphs. \\
	\end{abstract}
	
	
	\section{Introduction}
	Consider the wave equation
	\begin{equation}
		\begin{cases}
			u_{tt}(x,t) = c(x)^2\Delta u(x,t),
			& (x,t)\in\mathbb{R}^d\times(0,T),\\
			u(x,0) = f(x),
			& x\in\mathbb{R}^d,\\
			u_t(x,0) = 0,
			& x\in\mathbb{R}^d.
		\end{cases}
		\label{eq:wave-equation}
	\end{equation}
	where $c(x)>0$ is the wave speed and $f$ is the initial source. We assume that $f$ is supported in a bounded region $\Omega\subset\mathbb{R}^d$ and that $c$ is known outside $\Omega$. For a fixed observation time $T>0$, the full boundary measurement is
	\begin{equation*}
		\Lambda(f,c):=u|_{\partial\Omega \times (0, T)}.
	\end{equation*}
	More generally, if $\Sigma\subset\partial\Omega$ is some subset of the boundary, one may consider instead the partial measurement
	\begin{equation*}
		\Lambda_{\Sigma}(f,c):=u|_{\Sigma \times (0, T)}.
	\end{equation*}
	Despite substantial interest in applications such as medical imaging, the simultaneous recovery of a general initial source $f$ and a general spatially varying wave speed $c$ from a single passive measurement remains open in the multidimensional setting. Indeed, the linearized problem was shown by Stefanov and Uhlmann to be unstable at every scale of Sobolev spaces \cite{StefanovUhlmann2013Instability}. The stability theory available when either the source or speed is known cannot be expected to extend directly to the problem of simultaneous recovery. 
	
	There are nevertheless several compelling partial results. Liu and Uhlmann establish simultaneous recovery from a single passive boundary measurement under an admissibility condition and additional structural restrictions, including the case of an unknown constant sound speed and a source independent of one spatial coordinate \cite{LiuUhlmann2015}. Acosta proves simultaneous recovery under a relative-uncertainty condition controlling the admissible discrepancy in the sound speed by that of the initial pressure \cite{Acosta2019}. Knox and Moradifam derive integral identities separating information about the sound speed from the unknown source and obtain uniqueness under additional structural hypotheses \cite{KnoxMoradifam2020}. Kian and Uhlmann prove simultaneous determination of the sound speed and initial source under a monotonicity condition on the sound speed \cite{KianUhlmann2025}. Feizmohammadi employs inverse spectral methods to achieve simultaneous recovery in dimension one \cite{Feizmohammadi2025}. More recently, Kian and Liu establish uniqueness for broad classes of piecewise-constant sound speeds in three dimensions, and Hölder stability in a more restricted setting \cite{KianLiu2025}. Chen et al. develop several simultaneous recovery results for the wave and related hyperbolic problems under structural admissibility assumptions \cite{ChenJiangLiuLoTao2025}. Finally, Moradifam proves global uniqueness and Lipschitz stability under a constitutive relation linking the initial source to the sound speed, together with certain geometric assumptions \cite{Moradifam2026}.
	
	In the present work, we study a discrete analogue of this simultaneous recovery problem, motivated primarily by spatial discretizations of the wave equation. Writing the wave equation in the equivalent mass form
	\begin{equation*}
		m(x)u_{tt}(x,t)-\Delta u(x,t)=0, \qquad m(x)=c(x)^{-2},
	\end{equation*}
	let $\Gamma_h=(V_h,E_h)$ be a mesh graph with nodes $x_1^h,\ldots,x_{N_h}^h$. Let $U_h(t)\in\mathbb{R}^{N_h}$ denote the vector of discrete nodal unknowns, with
	\begin{equation*}
		\bigl(U_h(t)\bigr)_i \approx u(x_i^h,t).
	\end{equation*}
	A finite-difference or mass-lumped finite-element discretization then leads to a system of the form
	\begin{equation*}
		M_hU_{h,tt}(t)+L_hU_h(t) = 0,
	\end{equation*}
	where $M_h$ is a diagonal and $L_h$ is the discrete spatial operator. This motivates the following graph-theoretic model. Let $\Gamma=(V,E)$ be a graph with $N$ vertices and graph Laplacian $L$. Consider
	\begin{equation}
		\label{eq:intro-mass-form}
		\begin{cases}
			MX_{tt}(t)=-LX(t), & t \geq 0,\\
			X(0)=f, \\
			X_t(0)=0.
		\end{cases}
	\end{equation}
	where $f\in\mathbb{R}^N$ is the initial signal and
	\begin{equation*}
		M=\operatorname{diag}(m_1,\ldots,m_N)
	\end{equation*}
	contains the vertex masses.
	
	Within this framework, we have exhibited sufficient conditions for the recovery of the initial signal or the vertex masses from a single passive boundary measurement \cite{MoradifamWaite2026a}. In that work, the central mechanism allowing for recovery of a single parameter was the zero forcing condition. The present work asks what additional geometric redundancies might be required when both the initial signal and the vertex masses are unknown.
	
	The first strengthening of the zero forcing condition is the new coloring game we call $r$-witness forcing. In the minimal $2$-witness case, we show that if information is able to propagate through the graph in a redundant manner, so that each newly recovered vertex is forced simultaneously by two previously recovered vertices, then the two resulting equations can be used to separate the unknown mass discrepancies from the value being propagated. This gives simultaneous recovery of the vertex masses and initial signal from a single passive measurement under generic hypotheses.
	
	We then show that the information needed to disambiguate the masses need not be localized so strictly. In the frontier propagation framework, the redundant equations may instead be distributed across a collection of vertices forming a frontier between successive contour sets. The relevant information is encoded by the coupling matrix between the frontier and the next contour: full column rank permits propagation to the next layer, while a nontrivial left kernel supplies relations that eliminate the unknown propagated values and isolate the mass discrepancies. This gives a substantially more general simultaneous recovery mechanism, of which sequential $2$-witness forcing appears as a minimal special case. In particular, triangular funnel graphs admit redundant frontier propagation even though the observation set does not satisfy the $2$-witness zero forcing condition. The required nondegeneracy conditions are generic when $(L,M)$ has simple generalized spectrum, although spectral simplicity is not needed for the uniqueness result itself.
	
	For the frontier propagation mechanism, we establish injectivity of the Fréchet derivative and thereby Lipschitz stability on compact admissible parameter classes. Finally, we show that the uniqueness and stability results extend to positively weighted graphs under the corresponding weighted geometric hypotheses.
	
	The remainder of the paper is organized as follows. Section \ref{sec:background} establishes the requisite notation and definitions. Section \ref{sec:2-witness} contains the proof of the $2$-witness result together with examples illustrating its scope and limitations. Section \ref{sec:local-redundancy} proves uniqueness and Lipschitz stability for the more general frontier propagation mechanism, as well as a weakening of the nondegeneracy assumptions when the masses are constant along the contours of the graph. We conclude in Section \ref{sec:weighted} by extending the uniqueness and stability results to positively weighted graphs.
	
\section{Preliminaries}
	
	\label{sec:background}
	
	In this section, we introduce the graph-theoretic notation and the discrete wave models used throughout the paper. We also define graph-theoretic notions that will play a role in the various conditions for simultaneous recovery.
	
\subsection{Graph notation}
	
	We begin with the basic notation for finite graphs. Throughout the paper, all graphs are assumed to be finite, simple, undirected, and connected, with at least two vertices.
	
	\begin{defn}
		An ordered pair $\Gamma=(V,E)$ is a simple undirected graph if
		\begin{equation*}
			V=\{1,\ldots,N\}\subseteq\mathbb{N}
		\end{equation*}
		and
		\begin{equation*}
			E\subseteq\bigl\{\{i,j\}\mid i,j\in V,\ i<j\bigr\}.
		\end{equation*}
		The set $V$ is called the vertex set and $E$ the edge set.
	\end{defn}
	
	\begin{defn}
		The adjacency matrix $A_\Gamma\in\mathbb{R}^{N\times N}$ is defined by
		\begin{equation*}
			(A_\Gamma)_{ij}:=
			\begin{cases}
				1, & \text{if }\{i,j\}\in E,\\
				0, & \text{otherwise}.
			\end{cases}
		\end{equation*}
		The degree matrix is
		\begin{equation*}
			D_\Gamma=\operatorname{diag}(d_1,\ldots,d_N),
			\qquad
			d_i=\sum_{j=1}^N(A_\Gamma)_{ij}.
		\end{equation*}
	\end{defn}
	
	\begin{defn}
		The graph Laplacian is the matrix
		\begin{equation*}
			L=D_\Gamma-A_\Gamma.
		\end{equation*}
	\end{defn}
	
	For a simple graph, the off-diagonal entries of $L$ determine adjacency and the diagonal entries determine the vertex degrees. Thus $L$ determines the graph and will be regarded as encoding its geometry.
	
	A scalar graph signal on $\Gamma$ is a vector $X\in\mathbb{R}^N$, where $X_i$ denotes the value associated with vertex $i$. For $u\in V$, let $E_{uu}\in\mathbb{R}^{N\times N}$ denote the matrix with a $1$ in entry $(u,u)$ and zeros elsewhere.
	
	\begin{defn}
		Let $B\subseteq V$. Initially color the vertices in $B$ black and the remaining vertices white. If a black vertex has exactly one white neighbor, that neighbor may be changed to black. The set $B$ is called a zero forcing set if repeated application of this rule eventually colors every vertex black. A zero forcing sequence is then a sequence $(S_k)_{k=0}^K$ with $S_0=B$ and $S_K=V$ such that, at each stage, a vertex $u\in S_k$ forces a vertex $w\in S_{k+1}\setminus S_k$. In this case one may say that $S_k$ forces $S_{k + 1}$ and $u$ forces $w$. For a fixed zero forcing sequence $(S_k)_{k = 0}^K$, the set of forcing vertices is the collection 
		\begin{equation*}
			P := \{u \in V \; | \; \text{there is a } k \text{ such that } u \in S_k \text{ forces } w \in S_{k + 1} \setminus S_k\}.
		\end{equation*} 
	\end{defn}
	
	\begin{defn}
		Let $B \subseteq V$ be a zero forcing set of a graph $\Gamma$. Fix a zero forcing sequence $(S_k)_{k = 0}^K$. A forcing chain is a maximal sequence of distinct vertices 
		\begin{equation*}
			v_0, v_1, \dots, v_\ell
		\end{equation*}
		such that $v_i$ forces $v_{i + 1}$ for every $i \in \{0, \dots, \ell -1\}$. Vertices of $B$ that do not participate in a force are regarded as chains of length zero. The forcing chains are pairwise disjoint and partition $V$.
	\end{defn}
	
	The zero forcing condition and the auxiliary concepts here are standard and used extensively in the literature on inverse problems on graphs. A comprehensive account of this literature is given in \cite{HogbenLinShader2022}. The next definition is a new generalization of the zero forcing condition via the introduction of more stringent requirements on the forcing moves. 
	
	\begin{defn}
		Let $r \in \mathbb{N}$ and $B \subseteq V$. Initially color the vertices in $B$ black and the remaining vertices white. If $r$ distinct black vertices $u_1, \dots, u_r$ each have exactly one white neighbor and this neighbor is the same vertex $w$, then $w$ may be changed to black. The set $B$ is called an $r$-witness zero forcing set if repeated application of this rule eventually colors every vertex black. An $r$-witness zero forcing sequence is then a sequence $(S_k)_{k = 0}^K$ with $S_0 = B$ and $S_K = V$ such that, at each stage, distinct vertices $u_1, \dots, u_r \in S_k$ force a vertex $w \in S_{k + 1} \setminus S_k$, meaning that $w$ is the unique neighbor of each $u_i$ outside $S_k$. In this case one may say that $\{u_1, \dots, u_r\}$ witnesses the forcing of $w$.
	\end{defn}	
	
	\begin{exmp}
		Consider a finite rooted $r$-ary tree, so that every non-leaf vertex has exactly $r$ children, and take the set of leaves as the boundary set $B$. The leaves $r$-witness force their parents simultaneously. Once the entire next level has been colored, the $r$ children of each vertex in the level above have that vertex as their unique uncolored neighbor, and hence jointly force it. Continuing inductively, the forcing propagates level by level toward the root. Thus the set of leaves is an $r$-witness zero forcing set. This is illustrated schematically by Figure \ref{fig:r-ary-tree}.
		
		\begin{figure}[H]
\centering
    \begin{tikzpicture}[scale = 1.0]
        \coordinate (root) at (0,0);
    
        \coordinate (u1) at (-4,-1.8);
        \coordinate (u2) at (-1.5,-1.8);
        \coordinate (ur) at (4,-1.8);
    
        \coordinate (v11) at (-5,-3.6);
        \coordinate (v12) at (-4.4,-3.6);
        \coordinate (v1r) at (-3,-3.6);
    
        \coordinate (v21) at (-2.5,-3.6);
        \coordinate (v22) at (-1.9,-3.6);
        \coordinate (v2r) at (-0.5,-3.6);
    
        \coordinate (vr1) at (3,-3.6);
        \coordinate (vr2) at (3.6,-3.6);
        \coordinate (vrr) at (5,-3.6);
    
        \draw (root) -- (u1);
        \draw (root) -- (u2);
        \draw (root) -- (ur);
    
        \draw (u1) -- (v11);
        \draw (u1) -- (v12);
        \draw (u1) -- (v1r);
    
        \draw (u2) -- (v21);
        \draw (u2) -- (v22);
        \draw (u2) -- (v2r);
    
        \draw (ur) -- (vr1);
        \draw (ur) -- (vr2);
        \draw (ur) -- (vrr);
    
        \node[forced] at (root) {$2$};
    
        \node[forced] at (u1) {$1$};
        \node[forced] at (u2) {$1$};
        \node[forced] at (ur) {$1$};
    
        \node at (1.25,-1.8) {$\cdots$};
    
        \node[observed] at (v11) {};
        \node[observed] at (v12) {};
        \node at (-3.7,-3.6) {$\cdots$};
        \node[observed] at (v1r) {};
    
        \node[observed] at (v21) {};
        \node[observed] at (v22) {};
        \node at (-1.2,-3.6) {$\cdots$};
        \node[observed] at (v2r) {};
    
        \node at (1.25,-3.6) {$\cdots$};
    
        \node[observed] at (vr1) {};
        \node[observed] at (vr2) {};
        \node at (4.3,-3.6) {$\cdots$};
        \node[observed] at (vrr) {};
    
        \node[below = 4pt] at (v11) {$v_{11}$};
        \node[below = 4pt] at (v1r) {$v_{1r}$};
        \node[below = 4pt] at (v21) {$v_{21}$};
        \node[below = 4pt] at (v2r) {$v_{2r}$};
        \node[below = 4pt] at (vr1) {$v_{r1}$};
        \node[below = 4pt] at (vrr) {$v_{rr}$};
    \end{tikzpicture}
    \caption{A rooted $r$-ary tree of depth $2$.}
    \label{fig:r-ary-tree}
\end{figure}
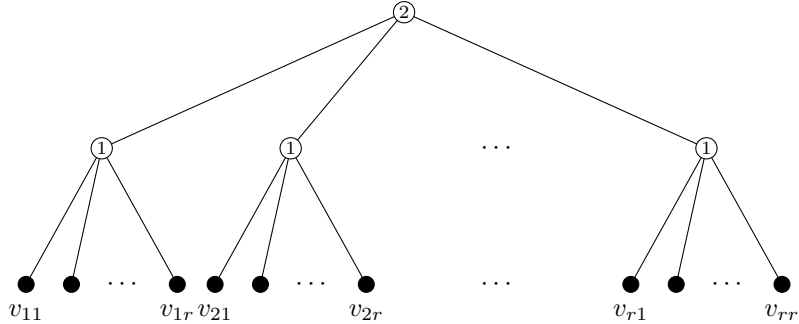
	\end{exmp}
	
\subsection{Wave notation}
	
	Throughout this paper, we denote the collection of $N \times N$ diagonal matrices by $\mathcal{D}^N$ and those with strictly positive entries by $\mathcal{D}_{++}^N$.
	
	\begin{defn}
		Let $\Gamma$ be a graph with graph Laplacian $L$, and let
		\begin{equation*}
			X:[0,\infty)\to\mathbb{R}^N
		\end{equation*}
		be a time-dependent graph signal. The graph wave equation is
		\begin{equation}
			\begin{cases}
				X_{tt}(t) = -C^2LX(t),
				& t\geq 0,\\
				X(0) = f, \\
				X_t(0) = 0,
			\end{cases}
			\label{eq:graph-wave-equation}
		\end{equation}
		where $C \in \mathcal{D}_{++}^N$ is the matrix of vertex wave speeds and $f\in\mathbb{R}^N$ is the initial signal.
	\end{defn}
	
	Equation \eqref{eq:graph-wave-equation} is a finite-dimensional system of ordinary differential equations in time, with the graph Laplacian describing the spatial coupling between vertices. Since it is a finite-dimensional system with constant coefficients, standard ODE theory guarantees a unique solution for every initial signal $f\in\mathbb{R}^N$ and every $C \in \mathcal{D}_{++}^N$. See, for example, \cite{ODE}.
	
	We work with a transformed version of \eqref{eq:graph-wave-equation}. Reparametrizing the vertex speeds by
	\begin{equation*}
		M=\diag(m_1,\ldots,m_N),
		\qquad
		m_i=c_i^{-2},
	\end{equation*}
	one may write
	\begin{equation}
		\label{eq:mass-form}
		\begin{cases}
			MX_{tt}(t)=-LX(t), & t \geq 0,\\
			X(0)=f, \\
			X_t(0)=0.
		\end{cases}
	\end{equation}
	We refer to this as the \textit{mass-form graph wave equation}.

\section{Recovery with 2-witness forcing}
\label{sec:2-witness}
	
	In this section, we exhibit the sufficiency of $2$-witness zero forcing for the simultaneous recovery of sources and speeds on the graph wave equation. Simultaneous recovery requires more information than the recovery of speeds or sources alone. As we shall demonstrate in the next section, this $2$-witness result can be viewed as the minimal or naïve case of a certain form of local information redundancy. If there are two witnesses to each forcing move, then at each step there is precisely enough redundant information to recover the masses at the forcing vertices and thereby propagate equality of the solutions to the newly forced vertex.
	
	Begin by introducing the set 
	\begin{equation*}
		\mathcal{G} := \left\{ f \in \mathbb{R}^N \mid X_{tt,u}^f(t) \text{ and } X_{tt,v}^f(t) \text{ are linearly independent on } [0, T] \text{ for all } u \neq v \in V \right\}
	\end{equation*}
	We will show in Section \ref{sec:local-redundancy} that, for $N \geq 3$ and simple generalized spectrum, this set is open, dense, and of full Lebesgue measure in $\mathbb{R}^N$.
	
\subsection{Uniqueness}
	
	\begin{thm}[Uniqueness under $2$-witness zero forcing and genericity hypotheses]
		\label{thm:2-witness-zero-forcing}
		
		Suppose that $B \subseteq V$ is a $2$-witness zero forcing set of $\Gamma$, fix a $2$-witness zero forcing sequence, and let $P \subseteq V$ be the corresponding set of forcing vertices. Fix $T > 0$, $f,\widetilde{f} \in \mathbb{R}^N$, and $M, \widetilde{M} \in \mathcal{D}_{++}^N$. Let $X$ and $\widetilde{X}$ denote the corresponding solutions of the mass-form graph wave equation. Suppose that $f \in \mathcal{G}$. If
		\begin{equation*}
			R_BX(t) = R_B\widetilde{X}(t)
			\qquad
			\text{for every } t \in [0,T],
		\end{equation*}
		then
		\begin{equation*}
			M = \widetilde{M}
			\qquad
			\text{and}
			\qquad
			f = \widetilde{f}.
		\end{equation*}
	\end{thm}
	
	\begin{proof}
		Suppose that $X_i(t) = \widetilde{X}_i(t)$ for every $i \in B$ and every $t \in [0, T]$. Set 
		\begin{equation*}
			W := X - \widetilde{X}. 
		\end{equation*}
		The equality of the observed traces gives immediately that 
		\begin{equation*}
			W_i \equiv 0
		\end{equation*}
		for each $i \in B$. Subtracting the two wave equations gives 
		\begin{equation}
			\label{eq:forced-wave}
			\widetilde{M}W_{tt} + LW = \left(\widetilde{M} - M\right) X_{tt}.
		\end{equation}
		
		Let $(S_k)_{k = 0}^K$ be the fixed $2$-witness zero forcing sequence with $S_0 = B$ and $S_K = V$. Suppose inductively that $W_j \equiv 0$ for each $j \in S_k$, and let $\{u, v\} \subseteq S_k$ witness the forcing of $w \notin S_k$. By the induction hypothesis that 
		\begin{equation*}
			W_u \equiv 0 \qquad \text{and} \qquad W_v \equiv 0, 
		\end{equation*}
		so that taking derivatives one has
		\begin{equation*}
			W_{tt, u} \equiv 0 \qquad \text{and} \qquad W_{tt, v} \equiv 0.
		\end{equation*}
		Since $w$ is the unique neighbor of $u$ outside $S_k$, all the other neighbors of $u$ belong to $S_k$. One has therefore that 
		\begin{equation*}
			\left(LW\right)_u = \deg(u) W_u - \sum_{j \sim u} W_j = -W_w.
		\end{equation*}
		Similarly, one has by evaluating at $v$ that 
		\begin{equation*}
			\left(LW\right)_v = \deg(v) W_v - \sum_{j \sim v} W_j = -W_w.
		\end{equation*}
		Thus from \eqref{eq:forced-wave},
		\begin{equation*}
			(\widetilde{m}_u - m_u) X_{tt, u} = -W_w = (\widetilde{m}_v - m_v) X_{tt, v}.
		\end{equation*}
		But recall that $f \in \mathcal{G}$. Thus $X_{tt, u}$ and $X_{tt, v}$ are linearly independent. Thus
		\begin{equation*}
			m_u = \widetilde{m}_u \qquad \text{and} \qquad m_v = \widetilde{m}_v,
		\end{equation*}
		and moreover $W_w \equiv 0$. By taking $t = 0$ one has $f_w = \widetilde{f}_w$. Thus equality of the traces and therefore equality of the sources propagate from $S_k$ to $S_{k + 1}$ and therefore by induction to all of $V$. Equality of the coefficients propagates similarly to all forcing witness vertices $i \in P$. 
		
		It remains to show that the masses are equal for vertices in $V \setminus P$. Therefore fix $j \in V \setminus P$. Since $X = \widetilde{X}$, subtracting the two equations at $j$ gives on the one hand that 
		\begin{equation*}
			(\widetilde{m}_j - m_j) X_{tt, j} = 0.
		\end{equation*}
		On the other hand,
		\begin{equation*}
			X_{tt,j} = -\frac{1}{m_j}(LX)_j.
		\end{equation*}
		Since $f \in \mathcal{G}$, $X_{tt,j} \not\equiv 0$ on $[0,T]$, and hence $m_j = \widetilde{m}_j$. Thus indeed 
		\begin{equation*}
			M = \widetilde{M}
			\qquad \text{and} \qquad
			f = \widetilde{f}
		\end{equation*}
		as claimed.
	\end{proof}
	
	\begin{exmp}
		Consider a $4 \times 4$ grid with observations made on three edges as in Figure \ref{fig:4x4-grid}.
		
		\begin{figure}[H]
\centering
    \begin{tikzpicture}[scale = 1.05]
        \foreach \i in {0,1,2,3} {
            \foreach \j in {0,1,2,3} {
                \coordinate (v\i\j) at (\j,-\i);
            }
        }
        
        \foreach \i in {0,1,2,3} {
            \foreach \j in {0,1,2} {
                \pgfmathtruncatemacro{\jp}{\j + 1}
                \draw (v\i\j) -- (v\i\jp);
            }
        }
        
        \foreach \i in {0,1,2} {
            \pgfmathtruncatemacro{\ip}{\i + 1}
            \foreach \j in {0,1,2,3} {
                \draw (v\i\j) -- (v\ip\j);
            }
        }
        
        \foreach \j in {0,1,2,3} {
            \node[observed] at (v0\j) {};
            \node[observed] at (v3\j) {};
        }
        
        \node[observed] at (v10) {};
        \node[observed] at (v20) {};
        
        \node[forced] at (v11) {$1$};
        \node[forced] at (v21) {$1$};
        \node[forced] at (v12) {$2$};
        \node[forced] at (v22) {$2$};
        \node[forced] at (v13) {$3$};
        \node[forced] at (v23) {$3$};
    \end{tikzpicture}
    \caption{A $2$-witness zero forcing set on the $4 \times 4$ grid.}
    \label{fig:4x4-grid}
\end{figure}
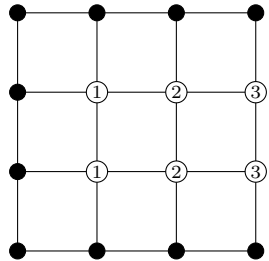
		
		The two vertices labeled $1$ may be forced simultaneously. Each is witnessed by the observed vertex directly above or below it and by the observed vertex immediately to its left. Once these vertices have been forced, they provide the second witnesses for the two vertices labeled $2$, while the corresponding vertices on the upper and lower observed edges provide the first witnesses. The same argument then forces the two vertices labeled $3$. Thus two parallel $2$-witness forcing chains propagate simultaneously across the two interior rows of the grid.
	\end{exmp}
	
	\begin{exmp}
		The previous example can be extruded along the third coordinate to produce a $2$-witness zero forcing of the $4 \times 4 \times n$ cubic grid. In this case, observations are made on two parallel faces and one of the remaining perpendicular faces. If $k$ denotes the third coordinate when the graph is realized through the standard embedding in $\mathbb{Z}^3 \subseteq \mathbb{R}^3$, then each cross section at fixed $k$ is identical to the $4 \times 4$ grid of the preceding example.
		
		The forcing may be performed simultaneously across all of these cross sections. At the first stage, all vertices corresponding to the vertices labeled $1$ in the preceding diagram are forced. Since the entire corresponding lines in the third coordinate are forced at the same stage, the additional neighbors arising from edges in the third coordinate are already colored when the next stage begins. Consequently, the vertices labeled $1$ may serve as witnesses for the vertices labeled $2$ in every cross section simultaneously. The same argument then propagates to the vertices labeled $3$. Thus the two-dimensional $2$-witness forcing pattern extends unchanged through the third coordinate.
	\end{exmp}
	
	\begin{rmk}
		The preceding construction does not extend directly to an $n \times n$ grid when $n \geq 5$. Indeed, suppose that the upper, lower, and left edges are observed. At the first stage, the vertices immediately to the right of the left edge and adjacent to the upper and lower edges may be forced. 
		
		When $n \geq 5$, however, there is at least one additional interior row separating the two forcing fronts. After the first pair of forces, each newly forced vertex has both the vertex immediately to its right and a vertex in the adjacent interior row as uncolored neighbors. It therefore cannot serve as a witness for the next vertex in its row. The vertices on the observed perpendicular edge still have unique uncolored neighbors, but there is no second witness for any of these vertices. Consequently, the $2$-witness forcing process stalls after the first stage.
	\end{rmk}
	
\section{Recovery with redundant frontier propagation}
\label{sec:local-redundancy}
	
	The preceding theorem uses two independent forcing relations at each step to separate the unknown mass discrepancy from the value being propagated to a newly recovered vertex. The next result isolates this mechanism from the combinatorics of zero forcing by allowing the required redundancy to be distributed across a frontier separating successive contour sets of a graph function. Sequential $r$-witness zero forcing will then appear as the special case of this more general frontier-propagation framework, in which case $2$-witness forcing provides the minimal version.
	
	\begin{defn} 
		Let $\Gamma = (V, E)$ be a graph. Let $\varphi: V \to \{0, 1, \dots, K\}$ be a graph function. Define the contour set
		\begin{equation*}
			C_k := \varphi^{-1}(k)
		\end{equation*}
		and set
		\begin{equation*}
			S_k = \bigcup_{i = 0}^k C_i.
		\end{equation*}
		
		Let $R_A: \mathbb{R}^N \to \mathbb{R}^{|A|}$ denote the restriction from $V$ onto $A$. For each $k$, define the eligible frontier
		\begin{equation*}
			\widehat{F}_k := \{u \in S_k \mid R_{\{u\}}LR_{C_{k + 1}}^T \neq 0 \text{ and } R_{\{u\}}LR_{V \setminus S_{k + 1}}^T  = 0\}.
		\end{equation*}
		A frontier set at level $k$ is a subset
		\begin{equation*}
			F_k \subseteq \widehat{F}_k.
		\end{equation*}
        
Intuitively, the eligible frontier $\widehat{F}_k$ consists of vertices in the already recovered region $S_k$ that are connected to the next contour $C_{k+1}$, but have no neighbors in any subsequent contour. Consequently, the wave equation at each eligible frontier vertex involves only previously recovered signal values and unknown values on $C_{k+1}$, allowing information to propagate one contour at a time. Associated with this choice, define the coupling matrix
		\begin{equation*}
			A_k := -R_{F_k}LR_{C_{k+1}}^T.
		\end{equation*}
		This is the adjacency matrix between the frontier set at level $k$ and the contour at level $k + 1$, and will be subject to geometric assumptions in the following theorem.
	\end{defn}

\subsection{Genericity of assumptions}
	
	We show that the analytic nondegeneracy conditions used in the frontier-propagation theorem are generic. First, simple generalized spectrum holds for an open, dense, full-measure set of positive diagonal mass matrices. For each such mass matrix, linear independence of the frontier accelerations holds for an open, dense, full-measure set of initial signals. The remaining nonharmonicity condition is also generic.
	
	\begin{lemma}
		The simple-spectrum condition is generic in the mass matrix and is sufficient for the genericity of the acceleration-independence conditions below. Indeed, identify $\mathcal{D}_{++}^N$ with $(0,\infty)^N$ and write
		\begin{equation*}
			\det(L - \lambda M) = \lambda p_M(\lambda),
		\end{equation*}
		where the factor $\lambda$ is simple since $\Gamma$ is connected. The discriminant
		\begin{equation*}
			\Delta(M) := \Disc(p_M)
		\end{equation*}
		is a polynomial in the diagonal entries of $M$. The discriminant detects repeated roots: it vanishes precisely when $p_M$ has a repeated root. Thus $(L,M)$ has simple generalized spectrum precisely when $\Delta(M) \neq 0$. To prove generic simplicity, it therefore suffices to show that $\Delta$ is not identically zero.
		
		To establish this, suppose that a positive generalized eigenvalue $\lambda$ has multiplicity $r \geq 2$, choose an $M$-orthonormal basis $\{\varphi_1,\dots,\varphi_r\}$ for its eigenspace and a vertex $i$ such that
		\begin{equation*}
			v_i := (\varphi_{1,i},\dots,\varphi_{r,i})^T \neq 0.
		\end{equation*}
		Under the diagonal perturbation $M_\varepsilon = M + \varepsilon E_{ii}$, standard first-order perturbation theory shows that the derivatives of the eigenvalue branches at $\varepsilon = 0$ are the eigenvalues of
		\begin{equation*}
			-\lambda v_i v_i^T.
		\end{equation*}
		Since $v_i \neq 0$, this rank-one matrix has one strictly negative eigenvalue and $r-1$ zero eigenvalues. Thus one branch separates from the repeated cluster to first order, reducing its multiplicity. Repeating this process reduces the multiplicities of all repeated clusters. At each step, the perturbation can be chosen sufficiently small to preserve the positive gaps between distinct eigenvalue clusters, so previously separated eigenvalues remain distinct. After finitely many steps, we obtain some $M_* \in \mathcal{D}_{++}^N$ with simple generalized spectrum. Hence $\Delta \not\equiv 0$, and its zero set has empty interior and Lebesgue measure zero. Therefore
		\begin{equation*}
			\{M \in \mathcal{D}_{++}^N \mid (L,M) \text{ has simple generalized spectrum}\}
		\end{equation*}
		is open, dense, and of full Lebesgue measure.
	\end{lemma}
	
	\begin{lemma}
		For each fixed $M \in \mathcal{D}_{++}^N$, let $\mathcal{J}$ denote the set of initial signals for which the frontier accelerations are linearly independent at every stage. When $(L,M)$ has simple generalized spectrum, this condition is generic in the initial signal. Fix such an $M$, and let $\{\varphi_\ell\}_{\ell = 0}^{N - 1}$ be an $M$-orthonormal basis of generalized eigenvectors satisfying
		\begin{equation*}
			L\varphi_\ell = \lambda_\ell M\varphi_\ell,
			\qquad
			0 = \lambda_0 < \lambda_1 < \cdots < \lambda_{N - 1}.
		\end{equation*}
		For a fixed frontier $F_k = \{u_1,\dots,u_r\}$, one has $r \leq N - 1$ since $F_k \subsetneq V$. The key spatial observation is that the nonzero generalized eigenvectors retain enough information to span the values on any proper subset of vertices. Indeed, since the zero generalized eigenvector is proportional to $\mathbf{1}$, $M$-orthogonality of the eigenbasis gives
		\begin{equation*}
			\operatorname{span}\{\varphi_1,\dots,\varphi_{N - 1}\}
			=
			\{x \in \mathbb{R}^N \mid (M\mathbf{1}) \cdot x = 0\}.
		\end{equation*}
		The Euclidean orthogonal complement of this span is $\operatorname{span}\{M\mathbf{1}\}$. If the restrictions of $\varphi_1,\dots,\varphi_{N-1}$ to $F_k$ did not span $\mathbb{R}^r$, there would be a nonzero vector $z$ supported on $F_k$ and orthogonal to all these eigenvectors. It would follow that $z \in \operatorname{span}\{M\mathbf{1}\}$, which is impossible: $F_k$ is a proper subset of $V$, whereas every entry of $M\mathbf{1}$ is strictly positive.
		
		Thus one may choose distinct indices $p_1,\dots,p_r \in \{1,\dots,N - 1\}$ such that
		\begin{equation*}
			\left(\varphi_{p_i,u_j}\right)_{i,j = 1}^r
		\end{equation*}
		is invertible. Writing
		\begin{equation*}
			f = \sum_{\ell = 0}^{N - 1}\alpha_\ell\varphi_\ell,
		\end{equation*}
		one has
		\begin{equation*}
			X_{tt}^f(t) = -\sum_{\ell = 1}^{N - 1}\lambda_\ell\alpha_\ell\cos(\sqrt{\lambda_\ell}t)\varphi_\ell.
		\end{equation*}
		Intuitively, distinct generalized eigenvalues give distinct temporal oscillations and their eigenvectors describe how those oscillations appear at the frontier vertices. A generic initial signal excites enough of these modes to distinguish the frontier accelerations. More precisely, since the positive generalized eigenvalues are distinct, the functions $\cos(\sqrt{\lambda_\ell}t)$, $1 \leq \ell \leq N-1$, are linearly independent on every interval $[0,T]$ with $T>0$. If $\alpha_{p_i} \neq 0$ for each $i$, any linear relation among $\{X_{tt,u}^f \mid u \in F_k\}$ forces a vector into the kernel of the invertible matrix $\left(\varphi_{p_i,u_j}\right)_{i,j=1}^r$; hence all its coefficients vanish. Therefore failure of the desired condition for any frontier is contained in a finite union of hyperplanes
		\begin{equation*}
			\{f \in \mathbb{R}^N \mid \langle f,\varphi_{p_i}\rangle_M = 0\}.
		\end{equation*}
		Therefore $\mathcal{J}$ is dense and of full Lebesgue measure. Openness follows from continuity of the corresponding $L^2(0,T)$ Gram determinants. When $N \geq 3$, applying the same argument to each two-element subset of $V$ and taking a finite intersection also shows that $\mathcal{G}$ is open, dense, and of full Lebesgue measure.
	\end{lemma}
	
	\begin{lemma}
		Finally, the set
		\begin{equation*}
			\mathcal{G}_{\text{top}} := \left\{f \in \mathbb{R}^N \mid (LX_M^f)_v \not\equiv 0 \text{ on } [0,T] \text{ for every } v \in V \setminus \bigcup_{k = 0}^{K - 1}F_k \right\}
		\end{equation*}
		is generic. The map
		\begin{equation*}
			\Psi_v : \mathbb{R}^N \to C([0,T])
		\end{equation*}
		by
		\begin{equation*}
			\Psi_v(f) := (LX_M^f)_v
		\end{equation*}
		is linear and continuous for each $v \in V \setminus \bigcup_{k = 0}^{K - 1}F_k$ because the solution of the graph wave equation depends linearly and continuously on its initial signal. Since the graph is connected and has at least two vertices, $\Psi_v(e_v)(0) = (Le_v)_v = \deg(v)>0$, where $e_v$ is the $v$th coordinate vector. Thus $\ker\Psi_v$ is a proper closed linear subspace of $\mathbb{R}^N$. Its complement is therefore open, dense, and of full Lebesgue measure. Taking the finite intersection over $v \in V \setminus \bigcup_{k = 0}^{K - 1}F_k$ proves the corresponding assertion for $\mathcal{G}_{\text{top}}$.
	\end{lemma}

Together, these lemmas show that, once the geometric frontier conditions are satisfied, the analytic nondegeneracy assumptions required for simultaneous recovery hold generically in the mass matrix and initial signal.

\subsection{Uniqueness}
	
	\begin{thm}
		\label{thm:frontier-propagation}
		Let $\varphi: V \to \{0, \dots, K\}$ be a surjective graph function. For each $k \in \{0, \dots, K-1\}$, fix a frontier set $F_k \subseteq \widehat{F}_k$ with associated coupling matrix $A_k$. Fix $T > 0$, $f,\widetilde{f} \in \mathbb{R}^N$, and $M, \widetilde{M} \in \mathcal{D}_{++}^N$. Let $X$ and $\widetilde{X}$ denote the corresponding solutions of the mass-form graph wave equation. Suppose that $f \in \mathcal{J}$.

Suppose that
\begin{equation}
    \label{eq:propagation-condition}
    \rank A_k = |C_{k+1}|
\end{equation}
and
\begin{equation}
    \ker A_k^T \neq \{0\}
\end{equation}
for each $k \in \{0,\dots,K-1\}$.

For each $k$, let $Q_k$ be any matrix whose columns
form a basis of $\ker A_k^T$. Assume further that
\begin{equation}
    \label{eq:nondegeneracy-condition}
    Q_k^T R_{F_k} E_{uu}\mathbf{1} \neq 0
\end{equation}
for every $u \in F_k$. Suppose also that
		\begin{equation}
			\label{eq:nondegeneracy-off-frontier-condition}
			(LX)_j \not\equiv 0
			\qquad
			\text{on } [0,T]
			\qquad
			\text{for every }
			j \in V \setminus \bigcup_{k = 0}^{K - 1}F_k.
		\end{equation}
		If in addition
		\begin{equation*}
			R_{S_0}X(t) = R_{S_0}\widetilde{X}(t)
			\qquad
			\text{for every } t \in [0,T],
		\end{equation*}
		then
		\begin{equation*}
			M = \widetilde{M}
			\qquad
			\text{and}
			\qquad
			f = \widetilde{f}.
		\end{equation*}
	\end{thm}
	
	The geometric conditions above have two complementary roles. The full column rank of \(A_k\) allows information to propagate to the next contour, while the nontrivial left kernel provides redundant relations among the frontier equations that eliminate the unknown values on that contour. The nondegeneracy condition ensures that every frontier vertex participates in at least one such relation, allowing its mass discrepancy to be identified. More precisely, for \(u\in F_k\), let \(e_u\) denote the corresponding coordinate vector in \(\mathbb R^{|F_k|}\). Since \(\ker A_k^T=(\operatorname{Ran}A_k)^\perp\), condition \(\eqref{eq:nondegeneracy-condition}\) is equivalent to \(e_u\notin\operatorname{Ran}A_k\). In particular, this condition is independent of the chosen basis \(Q_k\).
	
	\begin{proof}
		Suppose that
		\begin{equation*}
			R_{S_0}X(t) = R_{S_0}\widetilde{X}(t)
			\qquad
			\text{for every } t \in [0, T].
		\end{equation*}
		Set
		\begin{equation*}
			W := X - \widetilde{X}.
		\end{equation*}
		The equality of the observations on the zeroth level gives immediately that
		\begin{equation*}
			R_{S_0}W \equiv 0.
		\end{equation*}
		Subtracting the two wave equations gives the forced wave equation
		\begin{equation}
			\label{eq:frontier-forced-wave}
			\widetilde{M}W_{tt} + LW = (\widetilde{M} - M)X_{tt}.
		\end{equation}
		
		Suppose by way of induction that
		\begin{equation*}
			R_{S_k}W \equiv 0
		\end{equation*}
		for some $k \in \{0, \dots, K - 1\}$. Since $F_k \subseteq S_k$, one has
		\begin{equation*}
			R_{F_k}W \equiv 0
			\qquad
			\text{and}
			\qquad
			R_{F_k}W_{tt} \equiv 0.
		\end{equation*}
		Moreover, the definition of $F_k$ gives
		\begin{equation*}
			R_{F_k}LR_{V \setminus S_{k + 1}}^T = 0.
		\end{equation*}
		Decomposing $W$ according to
		\begin{equation*}
			V = S_k \sqcup C_{k + 1} \sqcup (V \setminus S_{k + 1}),
		\end{equation*}
		it follows that
		\begin{align*}
			R_{F_k}LW
			& = R_{F_k}LR_{S_k}^T R_{S_k}W + R_{F_k}LR_{C_{k + 1}}^T R_{C_{k + 1}}W + R_{F_k}LR_{V \setminus S_{k + 1}}^T R_{V \setminus S_{k + 1}}W \\
			& = R_{F_k}LR_{C_{k + 1}}^T R_{C_{k + 1}}W \\
			& = -A_k R_{C_{k + 1}}W.
		\end{align*}
		
		Restricting \eqref{eq:frontier-forced-wave} to $F_k$, the first term vanishes, and hence
		\begin{equation*}
			-A_k R_{C_{k + 1}}W = R_{F_k}(\widetilde{M} - M)X_{tt}.
		\end{equation*}
		Since $M$ and $\widetilde{M}$ are diagonal,
		\begin{equation*}
			\widetilde{M} - M =	\sum_{u \in V}(\widetilde{m}_u - m_u)E_{uu}.
		\end{equation*}
		Because $R_{F_k}E_{uu} = 0$ whenever $u \notin F_k$, one obtains
		\begin{equation*}
			-A_k R_{C_{k + 1}}W	= \sum_{u \in F_k} (\widetilde{m}_u - m_u)R_{F_k}E_{uu}X_{tt}.
		\end{equation*}
		Multiplying by $Q_k^T$ and using $Q_k^T A_k = 0$ gives
		\begin{equation*}
			0 =	\sum_{u \in F_k} (\widetilde{m}_u - m_u) Q_k^T R_{F_k}E_{uu}X_{tt}.
		\end{equation*}
		Since
		\begin{equation*}
			E_{uu}X_{tt} = X_{tt,u}E_{uu}\mathbf{1},
		\end{equation*}
		this becomes
		\begin{equation*}
			\sum_{u \in F_k} X_{tt,u} (\widetilde{m}_u - m_u) Q_k^T R_{F_k}E_{uu}\mathbf{1} = 0.
		\end{equation*}
		
		By the defining property of $\mathcal{J}$, the functions
		\begin{equation*}
			\left\{X_{tt,u} \mid u \in F_k \right\}
		\end{equation*}
		are linearly independent on $[0,T]$. Therefore
		\begin{equation*}
			(\widetilde{m}_u - m_u)	Q_k^T R_{F_k}E_{uu}\mathbf{1} =	0
		\end{equation*}
		for every $u \in F_k$. By the nondegeneracy condition \eqref{eq:nondegeneracy-condition},
		\begin{equation*}
			Q_k^T R_{F_k}E_{uu}\mathbf{1} \neq 0,
		\end{equation*}
		and hence
		\begin{equation*}
			m_u = \widetilde{m}_u \qquad \text{for every } u \in F_k.
		\end{equation*}
		
		It follows that
		\begin{equation*}
			A_k R_{C_{k + 1}}W = 0.
		\end{equation*}
		Since $A_k$ has full column rank,
		\begin{equation*}
			R_{C_{k + 1}}W \equiv 0.
		\end{equation*}
		Consequently,
		\begin{equation*}
			R_{S_{k + 1}}W \equiv 0,
		\end{equation*}
		which completes the induction step.
		
		Beginning with $R_{S_0}W \equiv 0$ and iterating through
		$k = 0, \dots, K - 1$, one obtains
		\begin{equation*}
			R_{S_K}W \equiv 0.
		\end{equation*}
		Since $S_K = V$, it follows that $W \equiv 0$ and therefore $X = \widetilde{X}$. Evaluating at $t = 0$ gives $f = X(0) = \widetilde{X}(0) = \widetilde{f}$.
		
		It remains to recover the masses off the frontier. Since $W \equiv 0$, equation \eqref{eq:frontier-forced-wave} reduces to
		\begin{equation*}
			(\widetilde{M} - M)X_{tt} = 0.
		\end{equation*}
		Thus, for every $j \in V \setminus \bigcup_{k = 0}^{K - 1}F_k$,
		\begin{equation*}
			(\widetilde{m}_j - m_j)X_{tt,j} = 0.
		\end{equation*}
		On the other hand,
		\begin{equation*}
			(LX)_j = -m_j X_{tt,j} \not\equiv 0
		\end{equation*}
		by \eqref{eq:nondegeneracy-off-frontier-condition}. Since $m_j > 0$, one has
		\begin{equation*}
			X_{tt,j} \not\equiv 0,
		\end{equation*}
		and consequently
		\begin{equation*}
			m_j = \widetilde{m}_j \qquad  \text{for every } j \in V \setminus \bigcup_{k = 0}^{K - 1}F_k.
		\end{equation*}
		Together with the mass identities obtained on the frontier sets, this gives	$M = \widetilde{M}$. Therefore
		\begin{equation*}
			M = \widetilde{M} \qquad \text{and} \qquad	f = \widetilde{f},
		\end{equation*}
		as desired.
	\end{proof}

	\begin{rmk}
		Let $\tau : V \to \{0, \dots, K\}$ be the zero forcing time function associated with a fixed sequential $2$-witness zero forcing sequence. Then
		\begin{equation*}
			C_0 = S_0
		\end{equation*}
		and
		\begin{equation*}
			C_{k + 1} = \{w_k\}
		\end{equation*}
		for each $k \in \{0, \dots, K - 1\}$, where $w_k$ is the vertex forced at stage $k$. Let $u_k$ and $v_k$ be the two witnesses forcing $w_k$. Since $w_k$ is the unique neighbor of both $u_k$ and $v_k$ outside $S_k$, one has
		\begin{equation*}
			u_k, v_k \in \widehat{F}_k.
		\end{equation*}
		Thus one may take
		\begin{equation*}
			F_k = \{u_k, v_k\}, \qquad C_{k + 1} = \{w_k\}.
		\end{equation*}
		Ordering the rows of $A_k$ by $u_k$ and $v_k$ gives
		\begin{equation*}
			A_k =
			\begin{pmatrix}
				1 \\
				1
			\end{pmatrix}.
		\end{equation*}
		Thus $A_k$ has full column rank and
		\begin{equation*}
			Q_k :=
			\begin{pmatrix}
				1 \\
				-1
			\end{pmatrix}
		\end{equation*}
		spans $\ker A_k^T$. Moreover,
		\begin{equation*}
			Q_k^T R_{F_k} E_{u_k u_k} \textbf{1} = 1 \qquad \text{and} \qquad Q_k^T R_{F_k} E_{v_k v_k} \textbf{1} = -1.
		\end{equation*}
		Hence the nondegeneracy condition \eqref{eq:nondegeneracy-condition} is satisfied at every stage.
		
		The genericity hypothesis of Theorem \ref{thm:2-witness-zero-forcing} implies that
		\begin{equation*}
			X_{tt, u_k} \qquad \text{and} \qquad X_{tt, v_k}
		\end{equation*}
		are linearly independent on $[0, T]$. Thus $f \in \mathcal{J}$ for the chosen frontier sets.
		
		Moreover, the union of the chosen frontier sets is precisely the set of forcing vertices:
		\begin{equation*}
			P = \bigcup_{k = 0}^{K - 1} F_k.
		\end{equation*}
		Since $f \in \mathcal{G}$, no vertex acceleration vanishes identically, and hence
		\begin{equation*}
			(LX)_j \not\equiv 0 \qquad \text{on } [0, T]
		\end{equation*}
		for every
		\begin{equation*}
			j \in V \setminus \bigcup_{k = 0}^{K - 1} F_k.
		\end{equation*}
		This is precisely the residual nondegeneracy condition in Theorem \ref{thm:frontier-propagation}. Consequently, Theorem \ref{thm:2-witness-zero-forcing} follows as a corollary of Theorem \ref{thm:frontier-propagation}.
		
		Finally, the conditions
		\begin{equation*}
			\rank A_k = |C_{k + 1}| \qquad \text{and} \qquad \ker A_k^T \neq \{0\}
		\end{equation*}
		imply that
		\begin{equation*}
			|F_k| > |C_{k + 1}|.
		\end{equation*}
		When $C_{k + 1}$ consists of a single vertex, the smallest possible frontier therefore consists of two vertices. In this sense, sequential $2$-witness zero forcing is the minimal nontrivial instance of the frontier-propagation mechanism.
	\end{rmk}
	
	The following example illustrates that Theorem \ref{thm:frontier-propagation} applies in substantially more generality than $r$-witness forcing. 
	
	\begin{exmp}
		\label{exmp:funnel}
		Consider a layered graph whose contour sets decrease in size by one at each stage. More precisely, let
		\begin{equation*}
			C_k = \{v_{k,0}, \dots, v_{k,m-k}\},
		\end{equation*}
		and join each vertex $v_{k + 1,j} \in C_{k + 1}$ to the two vertices $v_{k,j}$ and $v_{k,j + 1}$ of $C_k$. Assume that there are no edges between nonconsecutive contour sets. The resulting graph has the form of a triangular funnel, as illustrated in Figure \ref{fig:funnel}.
		
		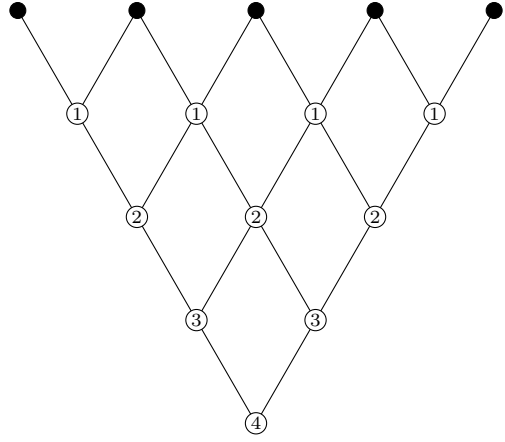
\begin{figure}[H]
\centering
    \begin{tikzpicture}[scale = 1.05]
        \coordinate (a0) at (-3,0);
        \coordinate (a1) at (-1.5,0);
        \coordinate (a2) at (0,0);
        \coordinate (a3) at (1.5,0);
        \coordinate (a4) at (3,0);
        
        \coordinate (b0) at (-2.25,-1.3);
        \coordinate (b1) at (-0.75,-1.3);
        \coordinate (b2) at (0.75,-1.3);
        \coordinate (b3) at (2.25,-1.3);
        
        \coordinate (c0) at (-1.5,-2.6);
        \coordinate (c1) at (0,-2.6);
        \coordinate (c2) at (1.5,-2.6);
        
        \coordinate (d0) at (-0.75,-3.9);
        \coordinate (d1) at (0.75,-3.9);
        
        \coordinate (e0) at (0,-5.2);
        
        \draw (a0) -- (b0);
        \draw (a1) -- (b0);
        \draw (a1) -- (b1);
        \draw (a2) -- (b1);
        \draw (a2) -- (b2);
        \draw (a3) -- (b2);
        \draw (a3) -- (b3);
        \draw (a4) -- (b3);
        
        \draw (b0) -- (c0);
        \draw (b1) -- (c0);
        \draw (b1) -- (c1);
        \draw (b2) -- (c1);
        \draw (b2) -- (c2);
        \draw (b3) -- (c2);
        
        \draw (c0) -- (d0);
        \draw (c1) -- (d0);
        \draw (c1) -- (d1);
        \draw (c2) -- (d1);
        
        \draw (d0) -- (e0);
        \draw (d1) -- (e0);
        
        \node[observed] at (a0) {};
        \node[observed] at (a1) {};
        \node[observed] at (a2) {};
        \node[observed] at (a3) {};
        \node[observed] at (a4) {};
        
        \node[forced] at (b0) {$1$};
        \node[forced] at (b1) {$1$};
        \node[forced] at (b2) {$1$};
        \node[forced] at (b3) {$1$};
        
        \node[forced] at (c0) {$2$};
        \node[forced] at (c1) {$2$};
        \node[forced] at (c2) {$2$};
        
        \node[forced] at (d0) {$3$};
        \node[forced] at (d1) {$3$};
        
        \node[forced] at (e0) {$4$};
    \end{tikzpicture}
    \caption{A layered funnel graph satisfying the hypotheses of Theorem \ref{thm:frontier-propagation}.}
    \label{fig:funnel}
\end{figure}
		
		Take $C_0$ as the observation set and let $\varphi$ assign to each vertex its contour index. Then $F_k = C_k$. After ordering the vertices in the natural way, the coupling matrix between successive contours is
		\begin{equation*}
			A_k =
			\begin{pmatrix}
				1 & 0 & \cdots & 0 \\
				1 & 1 & \ddots & \vdots \\
				0 & 1 & \ddots & 0 \\
				\vdots & \ddots & \ddots & 1 \\
				0 & \cdots & 0 & 1
			\end{pmatrix},
		\end{equation*}
		where $A_k$ has $|C_{k + 1}| + 1$ rows and $|C_{k + 1}|$ columns. In particular,
		\begin{equation*}
			\rank A_k = |C_{k + 1}|.
		\end{equation*}
		Moreover,
		\begin{equation*}
			\ker A_k^T = \text{span}\{(1,-1,1,-1,\dots)^T\}.
		\end{equation*}
		Since every entry of this left-kernel vector is nonzero, if $Q_k$ is chosen to have this vector as its single column, then
		\begin{equation*}
			Q_k^T R_{F_k} E_{uu} \textbf{1} \neq 0
		\end{equation*}
		for every $u \in F_k$. Thus the geometric hypotheses of Theorem \ref{thm:frontier-propagation} are satisfied at every stage.
		
		This example is not, in general, an instance of $r$-witness zero forcing for any $r \geq 2$. At the first stage, for example, the two vertices at the ends of $C_k$ each have a unique uncolored neighbor in $C_{k + 1}$, while the interior vertices of $C_k$ have two uncolored neighbors there. Hence no vertex of $C_{k + 1}$ has two witnesses in the sense of $2$-witness zero forcing. The redundancy is instead distributed across the entire frontier and is detected by the nontrivial left kernel of $A_k$. The example therefore illustrates the additional flexibility afforded by the matrix formulation of Theorem \ref{thm:frontier-propagation}.
	\end{exmp}
	
	We now consider a substantial weakening of the nondegeneracy hypotheses in Theorem \ref{thm:frontier-propagation} which is permitted in the case where the masses $m_i$ are constant along the contour sets $C_k$. Such a weakening is permitted since the mass discrepancies along each contour then reduce to a single scalar unknown. The resulting uniqueness statement therefore requires only one nontrivial left-kernel projection at each stage of the propagation.
	
	\begin{cor}[Uniqueness for contour-constant masses]
		\label{cor:level-set-uniqueness}
		Let $\varphi : V \to \{0, \dots, K\}$ be a surjective graph function, and suppose that, for each $k \in \{0, \dots, K - 1\}$, there is a frontier set $F_k \subseteq C_k$ such that
		\begin{equation*}
			\operatorname{rank} A_k = |C_{k + 1}|.
		\end{equation*}
		Fix $T > 0$, $f,\widetilde{f} \in \mathbb{R}^N$, and $M,\widetilde{M} \in \mathcal{D}_{++}^N$, and let $X$ and $\widetilde{X}$ denote the corresponding solutions of the mass-form graph wave equation. Suppose that, for every $k \in \{0, \dots, K\}$, there are constants $\mu_k,\widetilde{\mu}_k > 0$ such that
		\begin{align}
			R_{C_k}MR_{C_k}^T &= \mu_k I_{|C_k|}, \label{eq:constant-on-contours} \\
			R_{C_k}\widetilde{M}R_{C_k}^T &= \widetilde{\mu}_k I_{|C_k|}. \label{eq:constant-on-contours-tilde}
		\end{align}
		Suppose further that, for each $k \in \{0, \dots, K - 1\}$, there exists $q_k \in \ker A_k^T$ such that
		\begin{equation*}
			q_k^T R_{F_k}X_{tt} \not\equiv 0 \qquad \text{on } [0, T],
		\end{equation*}
		and that there exists $j \in C_K$ such that
		\begin{equation*}
			(LX)_j \not\equiv 0 \qquad \text{on } [0, T].
		\end{equation*}
		If
		\begin{equation*}
			R_{S_0}X(t) = R_{S_0}\widetilde{X}(t) \qquad \text{for every } t \in [0, T],
		\end{equation*}
		then
		\begin{equation*}
			M = \widetilde{M} \qquad \text{and} \qquad f = \widetilde{f}.
		\end{equation*}
	\end{cor}
	
	\begin{proof}
		Set $W := X - \widetilde{X}$ and suppose inductively that $R_{S_k}W \equiv 0$. Since $F_k \subseteq C_k$, restricting the difference equation to $F_k$ gives
		\begin{equation*}
			-A_kR_{C_{k + 1}}W = (\widetilde{\mu}_k - \mu_k)R_{F_k}X_{tt}.
		\end{equation*}
		Multiplying by $q_k^T$ gives
		\begin{equation*}
			0 = (\widetilde{\mu}_k - \mu_k)q_k^T R_{F_k}X_{tt},
		\end{equation*}
		so the nondegeneracy hypothesis implies $\mu_k = \widetilde{\mu}_k$. Hence $A_kR_{C_{k + 1}}W = 0$, and the full column rank of $A_k$ gives $R_{C_{k + 1}}W \equiv 0$. Iterating from $R_{S_0}W \equiv 0$ gives $W \equiv 0$, and therefore $f = \widetilde{f}$.
		
		It remains only to recover the mass on $C_K$. Since $W \equiv 0$, the difference equation at the chosen vertex $j \in C_K$ gives
		\begin{equation*}
			(\widetilde{\mu}_K - \mu_K)X_{tt,j} = 0.
		\end{equation*}
		Since $X_{tt,j} = -\mu_K^{-1}(LX)_j \not\equiv 0$, it follows that $\mu_K = \widetilde{\mu}_K$. Thus $M = \widetilde{M}$, completing the proof.
	\end{proof}

\subsection{Lipschitz stability}
	
	To obtain a Lipschitz stability estimate for the simultaneous recovery, we apply a theorem of Bourgeois. Fix a surjective graph function
	\begin{equation*}
		\varphi : V \to \{0, \dots, K\}
	\end{equation*}
	satisfying the geometric hypotheses of Theorem \ref{thm:frontier-propagation}, and let
	\begin{equation*}
		\mathcal{A} \subseteq \mathcal{D}_{++}^N \times \mathbb{R}^N
	\end{equation*}
	denote the set of pairs $(M, f)$ satisfying the hypotheses of that theorem. In particular, $f \in \mathcal{J}$, and
	\begin{equation*}
		(LX_M^f)_j \not\equiv 0
		\qquad
		\text{on } [0, T]
	\end{equation*}
	for every $j \in V \setminus \bigcup_{k = 0}^{K - 1}F_k$.
	
	The acceleration-independence and nonharmonicity conditions are open by continuous dependence on $(M,f)$, so $\mathcal{A}$ is open. Fix a distinguished parameter
	\begin{equation*}
		\theta_* = (M_*, f_*) \in \mathcal{A}.
	\end{equation*}
	There is a bounded open rectangle $\mathcal{V}$ containing $\theta_*$ such that
	\begin{equation*}
		\Theta := \overline{\mathcal{V}} \subseteq \mathcal{A}.
	\end{equation*}
	Thus $\Theta$ is compact and convex.
	
	Define the parameter-to-data map
	\begin{equation*}
		\Phi_T : \mathcal{A} \to L^2(0, T; \mathbb{R}^{|S_0|})
	\end{equation*}
	by
	\begin{equation*}
		\Phi_T(M, f) := R_{S_0}X_M^f.
	\end{equation*}
	The map $\Phi_T$ is smooth because the associated first-order system of ordinary differential equations depends smoothly on $M$ and $f$. Moreover, Theorem \ref{thm:frontier-propagation} implies that $\Phi_T$ is injective on $\mathcal{A}$.
	
	To compute its Fréchet derivative, fix
	\begin{equation*}
		\theta = (M, f) \in \mathcal{A},
	\end{equation*}
	a diagonal perturbation
	\begin{equation*}
		H = \diag(h_1, \dots, h_N),
	\end{equation*}
	and a perturbation
	\begin{equation*}
		g \in \mathbb{R}^N.
	\end{equation*}
	Define
	\begin{equation*}
		Z(t) = \left. \frac{d}{d \varepsilon} X_{M + \varepsilon H}^{f + \varepsilon g}(t) \right|_{\varepsilon = 0}.
	\end{equation*}
	Since
	\begin{equation*}
		X_{M + \varepsilon H}^{f + \varepsilon g}(0) = f + \varepsilon g
	\end{equation*}
	and the initial velocity vanishes for every $\varepsilon$, one has
	\begin{equation*}
		Z(0) = g,
		\qquad
		Z_t(0) = 0.
	\end{equation*}
	
	Differentiating
	\begin{equation*}
		(M + \varepsilon H)\left(X_{M + \varepsilon H}^{f + \varepsilon g}\right)_{tt} = -LX_{M + \varepsilon H}^{f + \varepsilon g}
	\end{equation*}
	at $\varepsilon = 0$ gives
	\begin{equation*}
		H(X_M^f)_{tt} + MZ_{tt} = -LZ.
	\end{equation*}
	Consequently,
	\begin{equation}
		\label{eq:frechet-derivative}
		D\Phi_T(\theta)[H, g] = R_{S_0}Z,
	\end{equation}
	where $Z$ is the solution of
	\begin{equation}
		\label{eq:variation-wave-equation}
		\begin{cases}
			MZ_{tt} + LZ = -H(X_M^f)_{tt}, \\
			Z(0) = g, \\ 
			Z_t(0) = 0.
		\end{cases}
	\end{equation}
	
	\begin{lemma}[Injectivity of the Fréchet derivative]
		\label{lemma:frechet-injectivity}
		For every $\theta \in \mathcal{A}$, the linear map
		\begin{equation*}
			D\Phi_T(\theta)
		\end{equation*}
		is injective.
	\end{lemma}
	
	\begin{proof}
		Suppose that
		\begin{equation*}
			D\Phi_T(\theta)[H, g] = 0.
		\end{equation*}
		By \eqref{eq:frechet-derivative},
		\begin{equation*}
			R_{S_0}Z = 0
		\end{equation*}
		in $L^2(0, T; \mathbb{R}^{|S_0|})$. Since $Z$ is continuous, it follows that
		\begin{equation*}
			R_{S_0}Z(t) = 0
			\qquad
			\text{for every } t \in [0, T].
		\end{equation*}
		In particular,
		\begin{equation*}
			R_{S_0}g = 0.
		\end{equation*}
		
		Proceed by way of induction. Suppose that
		\begin{equation*}
			R_{S_k}Z = 0.
		\end{equation*}
		Since $F_k \subseteq S_k$, one has
		\begin{equation*}
			R_{F_k}Z_{tt} = 0.
		\end{equation*}
		Moreover, by the definition of $F_k$ and the fact that $Z$ vanishes on $S_k$,
		\begin{align*}
			R_{F_k}LZ
			&= R_{F_k}LR_{C_{k + 1}}^T R_{C_{k + 1}}Z \\
			&= -A_k R_{C_{k + 1}}Z.
		\end{align*}
		Restricting \eqref{eq:variation-wave-equation} to $F_k$ therefore gives
		\begin{equation*}
			A_k R_{C_{k + 1}}Z
			=
			R_{F_k}H(X_M^f)_{tt}.
		\end{equation*}
		Since $H$ is diagonal,
		\begin{equation*}
			R_{F_k}H(X_M^f)_{tt} = \sum_{u \in F_k}	h_u \left(X_M^f\right)_{tt,u} R_{F_k}E_{uu}\mathbf{1}.
		\end{equation*}
		Multiplication by $Q_k^T$ now yields
		\begin{equation*}
			0 = \sum_{u \in F_k} h_u \left(X_M^f\right)_{tt,u} Q_k^T R_{F_k}E_{uu}\mathbf{1},
		\end{equation*}
		because the columns of $Q_k$ span $\ker A_k^T$ and hence
		\begin{equation*}
			Q_k^T A_k = 0.
		\end{equation*}
		
		By the defining property of $\mathcal{J}$, the accelerations 
		\begin{equation*}
			\left\{ \left(X_M^f\right)_{tt,u} \ \middle| \ u \in F_k \right\}
		\end{equation*}
		are linearly independent on $[0, T]$. It follows that
		\begin{equation*}
			h_u Q_k^T R_{F_k}E_{uu}\mathbf{1} = 0
		\end{equation*}
		for every $u \in F_k$. The nondegeneracy condition
		\begin{equation*}
			Q_k^T R_{F_k}E_{uu}\mathbf{1} \neq 0
		\end{equation*}
		therefore implies
		\begin{equation*}
			h_u = 0 \qquad \text{for every } u \in F_k.
		\end{equation*}
		Consequently,
		\begin{equation*}
			A_k R_{C_{k + 1}}Z = 0.
		\end{equation*}
		Since $A_k$ has full column rank,
		\begin{equation*}
			R_{C_{k + 1}}Z = 0.
		\end{equation*}
		Thus
		\begin{equation*}
			R_{S_{k + 1}}Z = 0.
		\end{equation*}
		
		Since $S_K = V$, the induction gives $Z \equiv 0$, and evaluating at $t = 0$ gives $g = 0$. It also follows from \eqref{eq:variation-wave-equation} that
		\begin{equation*}
			H(X_M^f)_{tt} = 0.
		\end{equation*}
		The induction has already shown that $h_u = 0$ for every $u \in \bigcup_{k = 0}^{K - 1}F_k$. For every
		\begin{equation*}
			v \in V \setminus \bigcup_{k = 0}^{K - 1}F_k,
		\end{equation*}
		the nonharmonicity hypothesis gives
		\begin{equation*}
			\left(X_M^f\right)_{tt,v} = -\frac{1}{m_v}(LX_M^f)_v \not\equiv 0,
		\end{equation*}
		and hence $h_v = 0$. Therefore $H = 0$, and thus $(H,g) = (0,0)$, which proves that $D\Phi_T(\theta)$ is injective.
	\end{proof}
	
	\begin{prop}[\cite{Bourgeois2013}, Theorem 2.1]
		\label{prop:bourgeois-theorem}
		Let $\mathcal{U}$ be an open subset of a finite-dimensional Hilbert space, let $\mathcal{K} \subseteq \mathcal{U}$ be compact and convex, and let
		\begin{equation*}
			G \in C^1(\mathcal{U}; \mathcal{H})
		\end{equation*}
		for a Hilbert space $\mathcal{H}$. If $G$ is injective on $\mathcal{U}$ and $DG(\theta)$ is injective for every $\theta \in \mathcal{U}$, then there exists $K_{\text{stab}} > 0$ such that
		\begin{equation*}
			\norm{\theta - \widetilde{\theta}}
			\leq
			K_{\text{stab}}
			\norm{G(\theta) - G(\widetilde{\theta})}
		\end{equation*}
		for every $\theta, \widetilde{\theta} \in \mathcal{K}$.
	\end{prop}
	
	Under the hypotheses of Theorem \ref{thm:frontier-propagation} $\Phi_T$ is injective on $\mathcal{A}$, while the Lemma \ref{lemma:frechet-injectivity} gives injectivity of its derivative. Therefore there exists $K_{\text{stab}} > 0$ such that
	\begin{equation}
		\label{eq:simultaneous-recovery-stability}
		\left(
		\norm{M - \widetilde{M}}_{\ell^2}^2
		+
		\norm{f - \widetilde{f}}_{\ell^2}^2
		\right)^{1/2}
		\leq
		K_{\text{stab}}
		\norm{
			\Phi_T(M, f)
			-
			\Phi_T(\widetilde{M}, \widetilde{f})
		}_{L^2(0,T;\mathbb{R}^{|S_0|})}
	\end{equation}
	for every $(M, f),(\widetilde{M},\widetilde{f}) \in \Theta$.
	
	Now suppose that the measured data are
	\begin{equation*}
		y_e = \Phi_T(M, f) + n,
	\end{equation*}
	where
	\begin{equation*}
		n \in L^2(0,T;\mathbb{R}^{|S_0|})
	\end{equation*}
	is the measurement noise. Let $(M_e, f_e) \in \Theta$ be a reconstruction and define its residual by
	\begin{equation*}
		r = \Phi_T(M_e, f_e) - y_e.
	\end{equation*}
	Then
	\begin{equation*}
		\Phi_T(M_e, f_e) - \Phi_T(M, f) = n + r.
	\end{equation*}
	Applying \eqref{eq:simultaneous-recovery-stability} and the triangle inequality gives
	\begin{equation}
		\label{eq:simultaneous-recovery-stability-noise}
		\left(
		\norm{M - M_e}_{\ell^2}^2
		+
		\norm{f - f_e}_{\ell^2}^2
		\right)^{1/2}
		\leq
		K_{\text{stab}}
		\left(
		\norm{n}_{L^2(0,T;\mathbb{R}^{|S_0|})}
		+
		\norm{r}_{L^2(0,T;\mathbb{R}^{|S_0|})}
		\right).
	\end{equation}
	
	The estimate can also be expressed in terms of the original vertex speeds. Since $\Theta$ is compact and contained in $\mathcal{D}_{++}^N \times \mathbb{R}^N$, there are constants
	\begin{equation*}
		0 < m_{\min} \leq m_{\max}
	\end{equation*}
	such that
	\begin{equation*}
		m_i,m_{e,i} \in [m_{\min},m_{\max}]
	\end{equation*}
	for every $(M, f),(M_e, f_e) \in \Theta$ and every $i \in V$. Since
	\begin{equation*}
		c_i = m_i^{-1/2},
	\end{equation*}
	the mean value theorem gives
	\begin{equation*}
		|c_i - c_{e,i}|
		\leq
		\frac{1}{2m_{\min}^{3/2}}
		|m_i - m_{e,i}|.
	\end{equation*}
	Consequently,
	\begin{equation}
		\label{eq:simultaneous-speed-recovery-stability}
		\norm{C - C_e}_{\ell^2}
		\leq
		\frac{K_{\text{stab}}}{2m_{\min}^{3/2}}
		\left(
		\norm{n}_{L^2(0,T;\mathbb{R}^{|S_0|})}
		+
		\norm{r}_{L^2(0,T;\mathbb{R}^{|S_0|})}
		\right).
	\end{equation}
	In particular, if the reconstruction fits the measured data exactly, then $r = 0$ and
	\begin{equation*}
		\norm{C - C_e}_{\ell^2}
		\leq
		\frac{K_{\text{stab}}}{2m_{\min}^{3/2}}
		\norm{n}_{L^2(0,T;\mathbb{R}^{|S_0|})}.
	\end{equation*}
	Since $m_{\min} = c_{\max}^{-2}$, this becomes
	\begin{equation}
		\label{eq:simultaneous-speed-recovery-stability-abstract}
		\norm{C - C_e}_{\ell^2}
		\leq
		\frac{K_{\text{stab}}}{2}c_{\max}^3
		\norm{n}_{L^2(0,T;\mathbb{R}^{|S_0|})}.
	\end{equation}
	
	\begin{rmk}
		Although \eqref{eq:simultaneous-speed-recovery-stability-abstract} has the same form as the stability estimate obtained in our earlier work on the recovery of speeds alone \cite{MoradifamWaite2026a}, the constant $K_{\text{stab}}$ need not be the same. Proposition \ref{prop:bourgeois-theorem} gives the existence of such a constant on a fixed compact parameter set, and the resulting constant depends implicitly on both the parameter set and the parameter-to-data map. Since the parameter-to-data map for simultaneous recovery differs from that for the recovery of speeds alone, there is no reason to expect the corresponding constants to agree.
		
		This distinction may become more significant under refinement toward a continuous model. Stefanov and Uhlmann showed that the linearized simultaneous recovery of the sound speed and source in the corresponding continuous multiwave problem is unstable in every scale of Sobolev spaces \cite{StefanovUhlmann2013Instability}, whereas they obtained stability estimates for the recovery of either the source or the speed separately under suitable hypotheses \cite{StefanovUhlmann2013Recovery}. This suggests that, for graph discretizations approximating such a continuous problem, the simultaneous recovery constants should not in general be expected to remain uniformly bounded under refinement, even when the corresponding single-parameter inverse problems remain stable. We emphasize, however, that no limiting statement of this kind is proved here. The uniqueness statements hold for every $T>0$: solutions of the finite-dimensional graph wave equation are real analytic in time, so agreement of observed traces on any nontrivial time interval implies agreement for all time. This does not imply uniform stability as $T\to 0$, and the stability constant may deteriorate for short observation intervals.
	\end{rmk}

\section{Extension to weighted graphs}
\label{sec:weighted}
	
	The simultaneous-recovery results extend directly to connected weighted graphs with known, fixed positive edge weights. Let $L_\omega$ denote the weighted graph Laplacian. In the mass-form parametrization, the weighted graph wave equation is
	\begin{equation*}
		MX_{tt} + L_\omega X = 0.
	\end{equation*}
	Let $X$ and $\widetilde{X}$ be the solutions corresponding to $(M, f)$ and $(\widetilde{M}, \widetilde{f})$, respectively, and set
	\begin{equation*}
		W := X - \widetilde{X}.
	\end{equation*}
	Subtracting the two weighted equations gives
	\begin{equation}
		\label{eq:weighted-forced-wave}
		\widetilde{M}W_{tt} + L_\omega W = (\widetilde{M} - M)X_{tt}.
	\end{equation}
	
	The frontier-propagation argument extends directly. Since all edge weights are positive, $L_\omega$ has the same off-diagonal zero pattern as $L$. The eligible frontier sets are therefore unchanged. For each chosen frontier set $F_k$, define the weighted coupling matrix
	\begin{equation*}
		A_{k, \omega} := -R_{F_k}L_\omega R_{C_{k + 1}}^T.
	\end{equation*}
	Restricting \eqref{eq:weighted-forced-wave} to $F_k$ therefore gives under the inductive hypothesis that
	\begin{equation*}
		-A_{k, \omega}R_{C_{k + 1}}W
		=
		\sum_{u \in F_k}
		(\widetilde{m}_u - m_u)X_{tt, u}R_{F_k}E_{uu}\textbf{1}.
	\end{equation*}
	
	Letting $Q_{k, \omega}$ be a matrix whose columns span $\ker A_{k, \omega}^T$, multiplication by $Q_{k, \omega}^T$ gives as before that
	\begin{equation*}
		0 = \sum_{u \in F_k} (\widetilde{m}_u - m_u)X_{tt, u} Q_{k, \omega}^T R_{F_k}E_{uu}\textbf{1}.
	\end{equation*}
	Thus, if the weighted accelerations
	\begin{equation*}
		\left\{X_{tt, u} \mid u \in F_k\right\}
	\end{equation*}
	are linearly independent and
	\begin{equation*}
		Q_{k, \omega}^T R_{F_k}E_{uu}\textbf{1} \neq 0
	\end{equation*}
	for every $u \in F_k$, the frontier induction therefore proceeds exactly as in Theorem \ref{thm:frontier-propagation}.
	
	After the induction gives $X = \widetilde{X}$, the remaining masses are recovered in the same manner as Theorem \ref{thm:frontier-propagation}. The theorem therefore remains valid after replacing $L$, $A_k$, and $\mathcal{J}$ by $L_\omega$, $A_{k, \omega}$, and the corresponding generic set $\mathcal{J}_\omega$, respectively. The contour-constant corollary extends by the same replacements. The weighted $2$-witness result follows as a special case as well.
	
	For a connected graph with positive edge weights, $L_\omega$ is symmetric, positive semidefinite, and satisfies
	\begin{equation*}
		\ker L_\omega = \linspan\{\textbf{1}\}.
	\end{equation*}
	The arguments establishing generic simplicity of the generalized spectrum, generic linear independence of the frontier accelerations, and generic nonharmonicity therefore apply with $L_\omega$ in place of $L$.
	
	Finally, the linearized equation becomes
	\begin{equation*}
		MZ_{tt} + L_\omega Z = -H(X_M^f)_{tt}.
	\end{equation*}
	For the frontier-propagation result, injectivity of the Fréchet derivative follows from the same weighted frontier induction. The corresponding Lipschitz stability estimates therefore remain valid in the weighted setting.
	
\FloatBarrier
\section*{Acknowledgments} Amir Moradifam is partially supported by NSF grant DMS-1953620.
	
\subsection*{Use of artificial intelligence}
	
	OpenAI's ChatGPT 5.6 Sol was used during the preparation of this work to assist with editorial refinement, identification of relevant literature, the preparation of figures. The main theoretical results were written by the authors, who also verified and edited all output and citations produced by artificial intelligence. The authors assume responsibility for all content.
	
\printbibliography
\end{document}